\documentclass[11pt]{amsart}
\usepackage{amsmath, amsthm, amssymb}
\usepackage{graphics}
\usepackage{natbib}
\usepackage{graphicx}
\usepackage{xcolor}
\usepackage[utf8]{inputenc}
\usepackage{quiver}
\usepackage{hyperref}
\usepackage{stmaryrd} 
\usepackage{yhmath} 

\newtheorem{theorem}{Theorem}[section]
\theoremstyle{definition} 
\newtheorem{definition}[theorem]{Definition} 
\theoremstyle{definition} 
\newtheorem{lemma}[theorem]{Lemma} 
\theoremstyle{definition} 
\newtheorem{proposition}[theorem]{Proposition} 
\theoremstyle{definition} 

\theoremstyle{definition} 
\newtheorem{corollary}[theorem]{Corollary} 
\theoremstyle{definition}

\theoremstyle{definition} 
 
\newtheorem*{acknowledgements}{Acknowledgements}

\newcommand{\NN}{\mathbb{N}}

\newcommand{\CC}{\mathbb{C}}
\newcommand{\m}{\mathfrak{m}}
\newcommand{\Spec}{\text{Spec }}
\newcommand{\p}{\mathfrak{p}}
\newcommand{\Quot}{\mathrm{Quot}}



\begin{document}

\title[deformations of isolated cyclic quotient singularities]{Deformations of isolated cyclic quotient singularities in arbitrary characteristic}

\author{Matthias Pfeifer}
\address{TU M\"unchen, Zentrum Mathematik - M11, Boltzmannstr. 3, 85748 Garching bei M\"unchen, Germany}
\email{m.pfeifer@tum.de}

\begin{abstract}
We show that toric surface singularities deform to toric surface singularities -  both in equal and mixed characteristic. As an application, we establish Riemenschneiders conjecture that isolated cyclic quotient singularities of any dimension deform to isolated cyclic quotient singularities in equal and mixed characteristic.
\end{abstract}

\maketitle
\setcounter{tocdepth}{1}
\tableofcontents

\section{Introduction}
In \cite{Rie74}, Riemenschneider calculated deformations of 2-dimensional cyclic quotient singularities over $\CC$ with low embedding dimension and conjectured that isolated cyclic quotient singularities deform to isolated cyclic quotient singularities. In dimension $\geq 3$ over $\CC$, this was already known, as isolated cyclic quotient singularities in dimension $\geq 3$ over $\CC$ are rigid by work of Schlessinger, \cite{Sch71}. The dimension 2 case over $\CC$ was proven by Koll{\'a}r and Shepherd-Barron in \cite{KSB88} by counting intersections of exceptional curves on a minimal resolution.

More generally, Riemenschneider additionally conjectured that finite quotient singularities deform to finite quotient singularities. This was proven over $\CC$ by Esnault and Viehweg in \cite{EV85}. In positive and mixed characteristic however, one needs to be careful. In \cite[Examle 3.10.8]{LMM25}, Liedtke, Martin and Matsumoto give an example of a quotient singularity deforming to a non-quotient singularity. They propose that the correct generalization of quotient singularities to positive and mixed characteristic in this context are linearly reductive quotient singularities. A slight reformulation of Riemenschneiders conjecture is given in \cite[Conjecture 2.12.1]{LMM25}, claiming that lrq singularities deform to lrq singualrities. This was proven by Sato and Takagi in \cite{ST21}.

This objective of this article is to establish Riemenschneiders conjecture on isolated cyclic quotient singularities in full generality, using cyclic lrq singularities.

\begin{definition}[\protect{\cite[Definition 2.6.5]{LMM25}}]

A \emph{cyclic linearly reductive quotient singularity}, or \emph{cyclic lrq singularity} for short, is a quotient singularity by the group-scheme $\mu_n$.
\end{definition}

Note that in characteristic zero this is equivalent to the classical definition of cyclic quotient singularities as quotients by finite cyclic groups. We refer to \cite{LMM25} for a detailed description of (cyclic) lrq singularities.

Our main result is the following.
\begin{theorem}[Theorem \ref{thm:cyc_to_cyc}]
Let $X$ be an isolated cyclic lrq singularity and let $\mathcal{X}$ be a proper deformation of $X$ over a complete DVR $S$. Then the geometric generic fiber $\mathcal{X}_{\overline{\eta}}$ of $\mathcal{X}\to \Spec S$ contains at worst isolated cyclic lrq singularities.
\end{theorem}
Note that if $S$ is of equal characteristic zero, this recovers the result from \cite{KSB88}.

The proof additionally shows semicontinuity of the embedding dimension and length of the associated group schemes.

\begin{proposition}[Proposition \ref{prop:semicontinuity}]
The embedding dimension as well as the length of the associated group schemes of singularities in $\mathcal{X}_{\overline{\eta}}$ is bounded above by the respective numbers for the singularity $X$.
\end{proposition}

Most of this article is phrased in terms of toric surface singularities. This is justified by the following proposition.

\begin{proposition}[Proposition \ref{prop:toric_eq_cyclic}; \protect{\cite[Proposition 2.7.4]{LMM25}}]
Let $x\in X$ be a two-dimensional isolated singularity over an algebraically closed field. Then the following are equivalent:
\begin{itemize}
\item{(1)} $x\in X$ is toric,
\item{(2)} $x\in X$ is a cyclic lrq singularity.
\end{itemize}
\end{proposition}

The proof of the main theorem uses singularity-deformations, defined in section \ref{sec:deformations}. This type of deformation allows us to reduce questions about singularities of proper deformations to questions about complete local rings. 

We refer to section \ref{sec:tor_surf_sings} and \ref{sec:deformations} for details on toric surface singularities and deformations respectively.

Section \ref{sec:tor_surf_sings} is a summary of facts about toric surface singularities and formulates a classification of isolated toric surface singularities used in the proof of the main theorem. This is followed by section \ref{sec:deformations}, where the notion of proper deformations and singularity-deformations are established. Section \ref{sec:shifting} is the technical heart of this article. There, we introduce automorphisms of a singularity-deformation of a toric surface singularity that allow us to greatly simplify the equations defining the deformation. These simplified equations allow us to explicitly calculate all singularities appearing in the generic fiber of a deformation, which is done in section \ref{sec:sing_of_def}, establishing the main theorem.\\

\begin{acknowledgements}
I am very grateful to my doctoral-advisor Christian Liedtke for introduction to this topic and his continued guidance and support. I would also like to thank Gebhard Martin for helpful discussion.
\end{acknowledgements}


\section{Toric Surface Singularities}\label{sec:tor_surf_sings}
In this section, we recall some basic facts about toric surface singularities and end with a theorem giving a new classification of isolated toric surface singularities.\\

By definition, a toric surface singularity is a singularity appearing in a toric surface, and any singularity formally isomorphic to it. We additionally assume the toric surface to be normal, so that it can be constructed in the usual way from a fan. This immediately implies that the singularity is isolated.

First, let us note that this paper can just as well be read as a paper about 2-dimensional isolated cyclic lrq singularities.
\begin{proposition}[\protect{\cite[Proposition 2.7.4]{LMM25}}]\label{prop:toric_eq_cyclic}
Let $x\in X$ be a two-dimensional isolated singularity over an algebraically closed field. Then the following are equivalent:
\begin{itemize}
\item{(1)} $x\in X$ is toric,
\item{(2)} $x\in X$ is a cyclic lrq singularity.
\end{itemize}
\end{proposition}

Indeed, many of the references you will find in this paper are phrased in terms of cyclic quotient singularities.\\

In \cite{Rie74}, where Riemenschneider conjectured that isolated cyclic quotient singularities deform to isolated cyclic quotient singularities, he also gave a description of the singularities in terms of equations and the relations between them.
\begin{theorem}[\cite{Rie74}]
Let $X$ be a toric surface singularity. Then there exist $e\in\NN$ and $a_{2},\dots,a_{e-1}\in\NN$ such that $X$ is formally isomorphic to
\begin{equation*}
\Spec k\llbracket  x_1,\dots,x_e\rrbracket /I,
\end{equation*}
where $I$ is the ideal generated by the $g_{i,j}$, defined in the following way
\begin{equation*}
g_{i,j}=\left\{
\begin{array}{ll} x_i x_j - x_{i+1}^{a_{i+1}}, &  i+1=j-1\\
x_i x_j - x_{i+1}^{a_{i+1}-1}(x_{i+2}^{a_{i+2}-2}\cdots x_{j-2}^{a_{j-2}-2})x_{j-1}^{a_{j-1}-1},& i+1<j-1,
\end{array}
\right.
\end{equation*}
for $1\leq i$ and $j\leq e$.

The module of relations between the generators is generated by

\begin{align*}
x_jg_{i,k}&=x_ig_{j,k}+\left(x_{j+1}^{a_{j+1}-2}\cdots x_{k-2}^{a_{k-2}-2}\right)x_{k-1}^{a_{k-1}-1}g_{i,j+1}&i<j<k-1\\
x_jg_{i,k}&=x_kg_{i,j}+x_{i+1}^{a_{i+1}-1}\left(x_{i+2}^{a_{i+2}-2}\cdots x_{j-1}^{a_{j-1}-2}\right)g_{j-1,k}&i+1<j<k.
\end{align*}
\end{theorem}

Note that $I$ is prime.

In order to simplify the notation, we will denote a singularity given by $a_2,\dots,a_{e-1}=\overline{a}$ by $X(\overline{a})$.

Note that one can make sense of the relations $g_{i,j}$ even if some of the $a_i$ are 0 or 1. If $a_i=0$ for some $i$, then all $x_i$ are invertible and $X(\overline{a})$ is regular, and $X(a_1,\dots,a_i,1,a_{i+1},\dots,a_e)\cong X(a_1,\dots,a_i-1,a_{i+1}-1,\dots,a_e)$.

We will concentrate primarily on the elements $g_{i,j}$ where $j = i+2$. A justification is given by the following theorem.

\begin{theorem}\label{thm:toric_eq}
Let $I\subseteq k\llbracket  x_1,\dots,x_e\rrbracket $ be a prime ideal of height $e-2$. If there exist $a_2,\dots,a_{e-1}\in\NN$, such that $x_{i-1}x_{i+1}-x_i^{a_i}\in I$ for all $i=2,\dots,e-1$, then either
\begin{itemize}
\item $I=(x_2,\dots,x_{e-1})$ or\\
\item $I=(\{g_{i,j}\,|\, 2\leq i+1\leq j-1\leq e-1\})$.\\[-3mm]
\end{itemize}
This implies that $\Spec k\llbracket  x_1,\dots,x_e\rrbracket /I$ is either smooth or a 2-dimensional toric surface singularity.
\end{theorem}
\begin{proof}
If $x_i\in I$ for some $i=1,\dots,e$, then by primality of $I$ we have $x_j\in I$ for all $j=2,\dots,e-1$. As $I$ has height $e-2$, we must have $I=(x_2,\dots,x_{e-1})$.

We now assume $x_i\notin I$ for all $i=1,\dots,e$. The relations between the $g_{i,j}$ tell us that for any $i$ and $j$ there exists some monomial $M$ such that $Mg_{i,j}\in I$. As $M\notin I$, we get $g_{i,j}\in I$.
\end{proof}


\section{Deformations}\label{sec:deformations}
In this section, we introduce the notion of a singularity-deformation and establish a few basic facts about them. We end this section by showing that questions about singularities of proper deformations can be reduced to questions about singularity-deformations.

Let $k$ be an algebraically closed field and let $X/k$ be a surface with at worst toric singularities. We are interested in the singularities arising in proper deformations of $X$ over a complete DVR $S$ with residue field $k$.

\begin{definition}
A \emph{proper deformation} of $X$ over $S$ is a cartesian diagram
\[\begin{tikzcd}[ampersand replacement=\&]
	X \& \mathcal{X} \\
	{\Spec k} \& {\Spec S,}
	\arrow[from=1-1, to=1-2]
	\arrow[from=1-1, to=2-1]
	\arrow["\pi"', from=1-2, to=2-2]
	\arrow[, from=2-1, to=2-2]
\end{tikzcd}\]
where $\pi$ is proper, flat and surjective.
\end{definition}

We will show that the geometric generic fiber $\mathcal{X}_{\overline{\eta}}$ contains at worst toric singularities. This will be done using the equations defining toric surface singularities. In order to work with the equations, we introduce a new type of deformation.

\begin{definition}
Let $X=k\llbracket  x_1,\dots,x_e\rrbracket /I$ and let $S$ be a complete DVR. A \emph{singularity-deformation} of $X$ over $S$ is given by a cartesian diagram
\[\begin{tikzcd}[ampersand replacement=\&]
	\Spec k\llbracket  x_1,\dots,x_e\rrbracket /I \& {\Spec S\llbracket  y_1,\dots,y_{e'}\rrbracket /I'} \\
	{\Spec k} \& {\Spec S,}
	\arrow[from=1-1, to=1-2]
	\arrow[from=1-1, to=2-1]
	\arrow["\pi"', from=1-2, to=2-2]
	\arrow["s", from=2-1, to=2-2]
\end{tikzcd}\]
where $\pi$ is flat and surjective.
\end{definition}

Before making a connection between the two types of deformations, we present a few lemmas about singularity-deformations.

\begin{lemma}\label{lem:emb_dim_doesnt_increase_sing}
In the above definition we may always assume $e'=e$.
\end{lemma}
\begin{proof}
The diagram above corresponds to a push-out diagram of rings
\[\begin{tikzcd}[ampersand replacement=\&]
	 k\llbracket  x_1,\dots,x_e\rrbracket /I \& { S\llbracket  y_1,\dots,y_{e'}\rrbracket /I'} \\
	{k} \& {S,}
	\arrow["\varphi", from=1-2, to=1-1]
	\arrow[from=2-1, to=1-1]
	\arrow[from=2-2, to=1-2]
	\arrow[from=2-2, to=2-1]
\end{tikzcd}\]
Choose $x_i\in\varphi^{-1}(x_i)\in S\llbracket y_1,\dots,y_{e'}\rrbracket$. If we can show that $S\llbracket x_1,\dots,x_e\rrbracket$ surjects onto $S\llbracket x_1,\dots,x_{e'}\rrbracket$, we are done. Let $z\in S\llbracket y_1,\dots,y_{e'}\rrbracket /I'$ be arbitrary and choose $r_\alpha\in k$ such that
\begin{align*}
\varphi(z)=\sum_{\alpha\in \NN^{e}}r_\alpha x^{\alpha}.
\end{align*}
Let $s_\alpha$ be preimages of the $r_\alpha$ in $S$. Then
\begin{align*}
z\in \sum_{\alpha\in \NN^{e}}s_\alpha x^{\alpha} + tz'
\end{align*}
for some $z'\in S\llbracket y_1,\dots,y_{e'}\rrbracket /I'$. We have found an element in $S\llbracket x_1,\dots,x_e\rrbracket$ whose image under the obvious inclusion morphism agrees with $y_{e'}$ in constant $t$-degree. Repeating the argument for $z'$ shows that we can find elements in $S\llbracket x_1,\dots,x_e\rrbracket$ that agree with $z$ up to any $t$-degree. It is easy to see that these elements form a converging series in $\in S\llbracket x_1,\dots,x_e\rrbracket$. 
\end{proof}

\begin{lemma}
Let $S\llbracket  x_1,\dots,x_n\rrbracket /I'$ be a singularity-deformation of $k\llbracket  x_1,\dots,x_n\rrbracket /I$ over a complete DVR $S$ and let $I$ be a finitely generated prime ideal. Then $I'$ is a prime ideal.
\end{lemma}
\begin{proof}
Fix a uniformizer $t$ of $S$ and let $ab\in I'$. Then for their images under the tensor product map $\overline{a},\overline{b}\in k\llbracket  x_1,\dots,x_n\rrbracket \cong k\otimes_{S}S\llbracket  x_1,\dots,x_n\rrbracket $ we have $\overline{a}\overline{b}\in I$ and therefore w.l.o.g. $\overline{a}\in I$.

As elements in $I$ lift to elements of $I'$ there exists $p\in S\llbracket  x_1,\dots,x_m\rrbracket $ such that
\begin{equation*}
a+tp\in I'.
\end{equation*}
This means that we have found an element in $I'$ that agrees with $a$ mod $(t)$.

Furthermore, we get
\begin{equation*}
tpb=(a+tp-a)b=(a+tp)b-ab\in I'.
\end{equation*}
As $S\llbracket  x_1,\dots,x_n\rrbracket /I'$ is flat over $S$, we have the implication
\begin{equation*}
tpb\in I'\quad\Rightarrow \quad pb\in I'.
\end{equation*}

We can now repeat this process either finding an element in $I'$ that agrees with $a$ up to $(t^2)$, or an element that agrees with $b$ up to $(t)$. Repeating this process indefinitely we either find a sequence of elements $a_n\in I'$ with $a_n\equiv a$ mod $(t^n)$ or such a sequence for $b$. W.l.o.g. let $a_n\in I$ be a sequence with $\lim_{n\in \NN}a_n=a$.

As $I$ is finitely generated, it's generators lift to generators of $I'$ and $I'$ is also finitely generated. As $S\llbracket  x_1,\dots,x_n\rrbracket $ is complete, any finitely generated ideal is closed. Since $I'$ is closed, the limit of any sequence of elements in $I'$ lies in $I'$ and therefore $a\in I'$.

We have shown
\begin{equation*}
ab\in I'\quad\Rightarrow a\in I'\,\lor\, a\in I'
\end{equation*}
meaning that $I'$ is prime.
\end{proof}

From now on, let $S$ denote a complete DVR with uniformizer $t$, and let $R=S\llbracket  x_1,\dots,x_e\rrbracket /I'$ denote a singularity-deformation of a toric surface singularity $X=k\llbracket  x_1,\dots,x_e\rrbracket /I$.

\begin{lemma}
All monomials in $R$ can uniquely be written as $x_ix_{i+1}+th$ where $h$ contains no terms divisible by $x_ix_j$ for $i-j\geq 2$.
\end{lemma}
\begin{proof}
Elements of $I$ lift to elements of $I'$, meaning that there exist $g_{i,j}'\in I'$ of the form
\begin{align*}
g_{i,j}'=g_{i,j}+th_{i,j}.
\end{align*}
We can therefore replace any occurrence of $x_ix_j$ with $i<j-1$ by a monomial in $x_{i+1},\dots,x_{j-1}$ plus some term of higher degree in $t$. Doing this first for all monomials of higher and higher $t$-degrees, the lemma follows from the completeness of $R$.
\end{proof}

\begin{definition}
We call replacing all monomials in $h_{i,j}$ by monomials in the ''compacted'' form \emph{compacting}.
\end{definition}

Note that compacting gives us a way of canonically representing any element in $R$ and that two elements of $R$ are equal if and only if their compacted forms are identical.\\

We will often look at elements of $R$ in order of their degrees in $t$. It will therefore prove useful to have a notion of divisibility that doesn't depend on higher degrees of $t$.

\begin{definition}
We say that a monomial $M$ is \emph{quasi-divisible} by $x_l$, if there exists some monomial $M'$ and formal power series $p$ such that $M=x_i\cdot M'+t^{\deg_t(M)+1}p$ in $R$. In other words, $M$ is quasi-divisible by $x_l$, if we can write it as a multiple of $x_l$ plus some terms of higher degree in $t$. If a monomial is not quasi-divisible by $x_l$, we say that it is \emph{quasi-constant} in $x_l$.
\end{definition}

Let $p$ be a formal power series. We can successively replace all monomials of $p$ that are quasi-divisible by $x_l$ by their counterpart that is a multiple of $x_l$ plus some terms of higher degree in $t$. We start in $t$-degree 0 and work our way up. This gives as a way to write $p$ as an element of $R$ such that all monomials in $p$ that are quasi-divisible by $x_l$ are truly multiples of $x_l$. This will be useful in the upcoming section, so we give this process a name.

\begin{definition}
Writing $p$ in this form will be called \emph{replacing $p$ by a maximally $x_l$-divisible representative}.
\end{definition}

\subsection{Relation between proper and singularity-deformations}
The following theorem tells us that if we're interested in singularities of a proper deformation of $X$, it is enough to look at singularity-deformations of the completions of local rings of points of $X$.

\begin{theorem}\label{thm:proper_to_sing_def}
Let $X$ have isolated singularities, let $\mathcal{X}\to\Spec S$ be a proper deformation of $X$, and assume that the geometric generic fiber $\mathcal{X}_{\overline{\eta}}$ has isolated singularities. Then all singularities of the geometric generic fiber $\mathcal{X}_{\overline{\eta}}$ appear as singularities of the geometric generic fiber of some singularity-deformation of points in $X$.
\end{theorem}
\begin{proof}
Let $y\in \mathcal{X}_{\overline{\eta}}$ be a closed point. We are interested in the singularity $\widehat{\mathcal{O}_{\mathcal{X},y}}$. Choose an affine open neighborhood $U=\Spec R$ of the image of $y$ in $\mathcal{X}$. Then $R$ is a finitely generated flat $S$-algebra, say $S[x_1,\dots,x_e]/I$. Let $K=\overline{\Quot(S)}$. We have
\[\begin{tikzcd}[ampersand replacement=\&]
	 \Spec k\times_{\Spec S} U \& U \& \Spec K\times_{\Spec S} U \ni y\\
	 X \& {\mathcal{X}} \& \mathcal{X}_{\overline{\eta}} \\
	 {\Spec k} \& {\Spec S} \& \Spec K,
	\arrow[from=1-1, to=1-2]
	\arrow[from=1-1, to=2-1]
	\arrow[from=1-2, to=2-2]
	\arrow[from=1-3, to=1-2]
	\arrow[from=1-3, to=2-3]
	\arrow[from=2-1, to=2-2]
	\arrow[from=2-1, to=3-1]
	\arrow[from=2-2, to=3-2]
	\arrow[from=2-3, to=2-2]
	\arrow[from=2-3, to=3-3]
	\arrow[from=3-1, to=3-2]
	\arrow[from=3-3, to=3-2]
\end{tikzcd}\]
with $U\to \Spec S$ flat and of finite type.

As the first and third line are all affine schemes, we get a diagram of rings
\[\begin{tikzcd}[ampersand replacement=\&]
	 k[x_1,\dots,x_e]/I \& S[x_1,\dots,x_e]/I \& K[x_1,\dots,x_e]/I \supseteq \p\\
	 {k} \& {S} \& K,
	\arrow["p_k",from=1-2, to=1-1]
	\arrow[from=2-1, to=1-1]
	\arrow[from=2-2, to=1-2]
	\arrow["p_K", from=1-2, to=1-3]
	\arrow[from=2-3, to=1-3]
	\arrow[from=2-2, to=2-1]
	\arrow[from=2-2, to=2-3]
\end{tikzcd}\]
where $\p$ denotes the maximal ideal corresponding to $y$, say $\p=(x_1-\lambda_1,\dots,x_e-\lambda_e)$. Then 
\begin{equation*}
\widehat{\mathcal{O}_{\mathcal{X}_{\overline{\eta}},y}}\cong \widehat{(K[x_1,\dots,x_e]/I)_\p}\cong K\llbracket x_1-\lambda_1,\dots,x_e-\lambda_e\rrbracket/I.
\end{equation*}

Choose $\m\subseteq k[x_1,\dots,x_e]/I$ maximal with $p_K^{-1}(\p)\subseteq p_k^{-1}(\m)$, i.e. $\m$ corresponds to a closed point $x\in X$ of the special fiber that is a specialization of $y$. 

Then $\widehat{(S[x_1,\dots,x_e]/I)_{p_k^{-1}(\m)}}\to S$ gives a singularity-deformation
\[\begin{tikzcd}[ampersand replacement=\&]
	 \widehat{\mathcal{O}_{X,x}}\cong \widehat{(k[x_1,\dots,x_e]/I)_\m} \& \widehat{(S[x_1,\dots,x_e]/I)_{p_k^{-1}(\m)}}\\
	 {k} \& {S,}
	\arrow[from=1-2, to=1-1]
	\arrow[from=2-1, to=1-1]
	\arrow[from=2-2, to=1-2]
	\arrow[from=2-2, to=2-1]
\end{tikzcd}\]
with geometric generic fiber
\begin{align*}
A:=\widehat{(S[x_1,\dots,x_e]/I)_{p_k^{-1}(\m)}}\otimes K \longleftarrow K.
\end{align*}
Since $p_K^{-1}(\p)\subseteq p_k^{-1}(\m)$, the morphism $S[x_1,\dots,x_e]/I\to K[x_1,\dots,x_e]/I$ induces a morphism $\widehat{(S[x_1,\dots,x_e]/I)_{p_k^{-1}(\m)}}\to \widehat{(K[x_1,\dots,x_e]/I)_\p}$ and further a morphism
\begin{align*}
\varphi:\, A\to \widehat{(K[x_1,\dots,x_e]/I)_\p}.
\end{align*}
Then $\varphi^{-1}(\p)$ is the maximal ideal $(x_1\otimes 1 - 1\otimes \lambda_1,\dots,x_e\otimes 1 - 1\otimes\lambda_e)\subseteq A$ and
\begin{align*}
\widehat{A_{\varphi^{-1}(\p)}} \cong K\llbracket x_1-\lambda_1,\dots,x_e-\lambda_e\rrbracket/I \cong \widehat{\mathcal{O}_{\mathcal{X}_{\overline{\eta}},y}}.
\end{align*}

We have therefore found $x\in X$ and a singularity-deformation of $\widehat{\mathcal{O}_{X,x}}$, such that the generic fiber of that deformation contains a point at which the completion of the local ring is isomorphic to $\widehat{\mathcal{O}_{\mathcal{X}_{\overline{\eta}},y}}$.
\end{proof}

Together with Lemma \ref{lem:emb_dim_doesnt_increase_sing}, we get
\begin{corollary}\label{cor:emb_dim_dosnt_increase_prop}
The embedding dimension of isolated singularities in the generic fiber of a proper deformation is bounded above by the embedding dimension of the isolated singularities of the special fiber.
\end{corollary}


\section{Shifting}\label{sec:shifting}
In this section, we introduce automorphisms of singularity-deformations of toric surface singularities as above that we call \emph{shifting}. These automorphisms will help us significantly simplify the equations defining such a deformation. After some tedious computations, the main result is $\ref{thm:shifting}$, which shows that after suitable isomorphism, $I'$ contains elements $g_{i-1,i+1}'$ of the form
\begin{equation*}
g_{i-1,i+1}'=x_{i-1}(x_{i+1}+tc_{i+1})-x_i^{a_i}+tx_ih_i
\end{equation*}
for some $c_{i+1}\in S$ and $h_i\in S[x_i]$, $\deg_{x_i} h_i < a_i$.\\
This will be used in \ref{sec:sing_of_def} to calculate the singularities of the deformation.

Given a deformation of a toric surface singularity
\[\begin{tikzcd}[ampersand replacement=\&]
	\Spec k\llbracket  x_1,\dots,x_e\rrbracket /I \& {\Spec S\llbracket  x_1,\dots,x_e\rrbracket /I'} \\
	{\Spec k} \& {\Spec S,}
	\arrow[from=1-1, to=1-2]
	\arrow[from=1-1, to=2-1]
	\arrow["\pi"', from=1-2, to=2-2]
	\arrow["s", from=2-1, to=2-2]
\end{tikzcd}\]
we know that $I$ contains elements $g_{i-1,i+1}:=x_{i-1}x_{i+1}-x_i^{a_i}$. This implies that $I'$ contains elements
\begin{align*}
g_{i-1,i+1}'=g_{i-1,i+1}+th_{i-1,i+1}.
\end{align*}

Our goal is to successively simplify the $h_{i-1,i+1}$ by variable shifts $x_{i+1}\mapsto x_{i+1}+tc$, where $c$ is chosen such that $h_{i-1,i+1}$ becomes simpler. This might introduce some new factors notably in $g_{i,i+2}'$ and $g_{i+1,i+3}$, possibly making them more complicated. This ''shifts'' the problem of simplifying $g_{i-1,i+1}'$ to that of simplifying $g_{i,i+2}$ and $g_{i+1,i+3}$ and will be called \emph{shifting to the right at $i$}. Similarly, we will also introduce shifting to the left. 

In order to define shifting to the right at $i$, we first replace $h_{i-1,i+1}$ by a maximally $x_{i-1}$-divisible representative. We can then write $g_{i-1,i+1}'$ as 
\begin{equation*}
g_{i-1,i+1}'=x_{i-1}(x_{i+1}+th_{i-1,i+1}')-x_i^{a_i}+th_{i-1,i+1}^++th_{i-1,i+1}^-,
\end{equation*}
where $h_{i-1,i+1}^+$ is constant in all $x_j$ for $j\leq i-1$ and $h_{i-1,i+1}^-$ is constant in all $x_j$ for all $j\geq i-1$.\\

In order to shorten notation, write $R:=S\llbracket  x_1,\dots,x_e\rrbracket /I'$.

\begin{lemma}
Let $\varphi$ be the morphism $R\to R$ defined by sending $x_j\mapsto x_j$ for all $j\neq i+1$ and $x_{i+1}\mapsto x_{i+1}+x_{i+1}u_1+u_2$ where $u_1$ and $u_2$ are elements of the maximal ideal of $R$ . Then $\varphi$ is an isomorphism.
\end{lemma}
\begin{proof}
It is clear that $\varphi$ is well-defined and injective. For surjectivity it is enough to show that $x_{i+1}$ is in the image. Assume that there exists no element in the image of $\varphi$ that coincides with $x_{i+1}$ up to $t$-degree $m$ and let $m$ be minimal with this property. Then there exists some $y\in R$ with $\varphi(y)-x_{i+1}=u\in (t^m)$ and we have
\begin{align*}
\varphi(y-u)-x_{i+1}=\varphi(y)-x_{i+1}-\varphi(u) = u-\varphi(u).
\end{align*}
As every element of $R$ coincides with its image under $\varphi$ at least up to the minimal $t$-degree among its monomials, we have $u-\varphi(u)\in (t^{m+1})$. We have found an element in the image of $\varphi$ coinciding with $x_{i+1}$ up tp $t$-degree $m$, a contradiction. This shows that $\varphi$ is surjective.
\end{proof}

This lemma allows us to take the inverse of the morphism $x_{i+1}\mapsto x_{i+1}+th_{i-1,i+1}'$, which will send $g_{i-1,i+1}'$ to $x_{i-1}x_{i+1}-x_i^{a_i}+th_{i-1,i+1}^++th_{i-1,i+1}^-$. Applying this morphism is what we will call \emph{shifting to the right at $i$} and write it as $\varphi_i^+$. Similarly, we can also \emph{shift to the left at $i$}, which we denote by $\varphi_i^-$.

Let us calculate the low-$t$-degree parts of $\varphi_i^+(x_{i+1})$. As described above we have
\begin{equation*}
\varphi_i^+(x_{i+1}-th_{i-1,i+1}')-x_{i+1}=t(h_{i-1,i+1}'-\varphi_i^+(h_{i-1,i+1}'))\in (t^{2 + \deg_t^{\min} h_{i-1,i+1}'}).
\end{equation*}
Therefore $x_{i+1}-th_{i-1,i+1}'$ coincides with $\varphi_i^+(x_{i+1})$ at least up to $(t^{2+\deg_t^{\min} h_{i-1,i+1}'})$. Furthermore, if $h_{i-1,i+1}'$ is not divisible by $x_{i+1}$, it is clear that the equation
$\varphi_i^+(x_{i+1})=x_{i+1}-th_{i-1,i+1}'$ holds in $R$.
Combining these to facts we get that 
\begin{equation*}
\varphi_i^+(x_{i+1})=x_{i+1}-th_{i-1,i+1}'
\end{equation*}
holds up to the minimal $t$-degree among monomials in $h_{i-1,i+1}'$ that are divisible by $x_{i+1}$.\footnote{It even holds up to $t$-degree 2 more then that, but we will not need that and it makes notation even harder.}

We will first shift to the right at $i=2,3,4,\dots,e-1$ (''fully shifting to the right'') and then shift to the left at $e-1,e-2,\dots,2$ (''fully shifting to the left''). 
Our goal is to show that repeatedly fully shifting to the right and the left converges and that the resulting $g_{i-1,i+1}'$ are of a simple form.

In order to say something about how fully shifting both ways changes the $g_{i-1,i+1}$, let us look at how shifting to the right at $i$ influences $g_{i,i+2}$ and $g_{i+1,i+3}$.

Let $m_i$ be the minimal $t$-degree among monomials in $x_{i+1}+th_{i-1,i+1}'$. Then
\begin{align*}
\varphi_i^+(g_{i,i+2}') &= \varphi_i^+(x_ix_{i+2}-x_{i+1}^{a_{i+1}}+th_{i,i+2})\\
&=x_ix_{i+1}-\varphi_i^+(x_{i+1})^{a_{i+1}}+t\varphi_i^+(h_{i,i+2})\\
&\in x_ix_{i+1}-(x_{i+1}-th_{i-1,i+1}'+(t^{1+m_i}))^{a_{i+1}}+th_{i,i+2}+(t^{1+m_i})\\
&= x_ix_{i+1}-x_{i+1}^{a_{i+1}}-a_{i+1}x_{i+1}^{a_i-1}th_{i-1,i+1}'+th_{i,i+2}+(t^{1+m_i})\\
&= x_ix_{i+1}-x_{i+1}^{a_{i+1}}+th_{i,i+2}-a_{i+1}x_{i+1}^{a_i-1}th_{i-1,i+1}'+(t^{1+m_i}).
\end{align*}
This shows that up to $t$-degree $m_i$, the image of $g_{i,i+2}'$ differs from $g_{i,i+2}$ only by $a_{i+1}x_{i+1}^{a_i-1}th_{i-1,i+1}'$.
Furthermore,
\begin{align*}
\varphi_i^+(g_{i+1,i+3}') &= \varphi_i^+(x_{i+1}x_{i+3}-x_{i+2}^{a_{i+2}}+th_{i+1,i+3})\\
&=\varphi_i^+(x_{i+1})x_{i+3}-x_{i+2}^{a_{i+2}}+t\varphi_i^+(h_{i,i+2})\\
&\in(x_{i+1}-th_{i-1,i+1}'+(t^{1+m_i}))x_{i+3}-x_{i+2}^{a_{i+2}}+th_{i,i+2}+(t^{1+m_i})\\
&=x_{i+1}x_{i+3}-th_{i-1,i+1}'x_{i+3}-x_{i+2}^{a_{i+2}}+th_{i,i+2}+(t^{1+m_i})\\
&=x_{i+1}x_{i+3}-x_{i+2}^{a_{i+2}}+th_{i,i+2}-th_{i-1,i+1}'x_{i+3}+(t^{1+m_i})
\end{align*}

This shows that up to $t$-degree $m_i$, the image of $g_{i+1,i+3}'$ differs from $g_{i+1,i+3}$ only by $-th_{i-1,i+1}'x_{i+3}$.\\

Let us now work through fully shifting both ways.

We start with $g_{i-1,i+1}'=x_{i-1}x_{i+1}-x_i^{a_i}+th_{i-1,i+1}$ for all $i=2,3,\dots,e-1$. After fully shifting to the right, all the $h_{i-1,i+1}$ contain no term quasi-divisible by $x_{i-1}$.

We now shift to the left at $e-1$. Then $h_{e-2,e}$ still doesn't contain any term quasi-divisible by $x_{e-2}$ or $x_e$. However, we have added a term of the form $a_{e-2}x_{e-2}^{a_{e-2}-1}th_{e-2,e}'$ to $h_{e-3,e-1}$, where $x_eh_{e-2,e}'$ are the terms in $h_{e-2,e}$ that were multiples of $x_e$ (the ''parts that got shifted to the left''). We have also added some terms of higher $t$-degree, but we will ignore them for now.

After shifting to the left at $e-2$, does $h_{e-3,e-1}$ contain any terms of $t$-degree 1 quasi-divisible by $x_{e-3}$ (terms that could be shifted to the right)? Before shifting to the left at $e-1$, it didn't, so we only need to consider what changed, i.e. $a_{e-2}x_{e-2}^{a_{e-2}-1}th_{e-2,e}'$. We know that no term in $x_eh_{e-2,e}'$ is quasi-divisible by $x_{e-2}$, as we had previously shifted to the right.  This implies that $h_{e-2,e}'$ is constant in all $x_j$ with $j\leq e-2$. Therefore, every monomial in $a_{e-2}x_{e-2}^{a_{e-2}-1}th_{e-2,e}'$ is constant in $x_j$ for all $j\leq e-3$, showing that there are no terms in $h_{e-3,e-1}$ constant in $t$ that could be shifted to the right.

Let us now shift to the left at $e-3$ and ask whether $h_{e-4,e-2}$ then contains any terms constant in $t$ that could be shifted to the right, i.e. terms quasi-divisible by $x_{e-4}$. After we first fully shifted to the right, all parts of $h_{e-4,e-2}$ were constant in $x_{e-4}$, so that we only have to look at what got added by shifting to the left at $e-1$ and $e-2$. The corresponding terms are $-th_{e-2,e}'x_{e-4}$ and $a_{e-3}x_{e-3}^{a_{e-3}-1}th_{e-3,e-1}'$. Similar to before, $a_{e-3}x_{e-3}^{a_{e-3}-1}th_{e-3,e-1}'$ can't add anything quasi-divisible by $x_{e-4}$. On the other hand, $-th_{e-2,e}'x_{e-4}$ is clearly divisible by $x_{e-4}$ and could be shifted to the right, so that the question to ask becomes which parts of $-th_{e-2,e}'x_{e-4}$ will remain after shifting to the left at $x_{e-3}$? Only those parts constant in $x_{e-2}$. We have previously established that $h_{e-2,e}'$ is constant in all $x_j$ with $j\leq e-2$. If any monomial in $h_{e-2,e}'$ is non-constant in any $x_j$ with $j\geq e-2$, then multiplying the monomial with $x_{e-4}$ will give something quasi-divisible by $x_{e-2}$ and will be shifted away to the left. Therefore, the only terms of $h_{e-4,e-2}$ constant in $t$ and quasi-divisible by $x_{e-4}$ are of the form $cx_{e-4}$, where $c\in k$. Note that the $t$-constant part of $h_{e-4,e-2'}$ (what is shifted to the left) is constant in $x_j$ for all $j\leq e-4$. 

An inductive argument shows that after fully shifting both ways, the only parts of $h_{i-1,i+1}$ constant in $t$ and quasi-divisible by $x_{i-1}$ are of the form $c_{i+1}x_{i-x}$ with $c\in k$.
One quickly checks that these elements will always be shifted back and forth with further shifting, but don't influence the shifting otherwise.\\

We have reached the following intermediate result.
\begin{lemma}
After fully shifting both ways, the $g_{i-1,i+1}'$ are of the form
\begin{align*}
g_{i-1,i+1}'=x_{i-1}(x_{i+1}+tc_{i+1})-x_i^{a_i}+tx_ih_i+t^2H_{i-1,i+1}
\end{align*}
for all $i=2,3,\dots,e-1$, where $c_{i+1}\in k$ and $h_i\in k[x_i]$ with $\deg h_i < a_i$.
\end{lemma}
\begin{proof}
We have established above that after fully shifting right and then left, the $t$-constant terms of $h_{i-1,i+1}$ are constant in $x_{i+1}$ and the only term quasi-divisible by $x_{i-1}$ is of the form $x_{i-1}c_{i+1}$.

It is left to show that any $t$-constant term of $h_{i-1,i+1}$ is constant in $x_j$ for $j<i-1$ and $j>i+1$.
Assume that there exists some $l$ such that $h_{l-1,l+1}$ contains a term constant in $t$ and non-constant in $x_k$ with $k>l+1$ and let $l$ be maximal with this property. Consider the element $x_{l+2}g_{l-1,l+1}'+x_l^{a_l-1}g_{l,l-2}'-x_{l+1}g_{l-1,l+2}'=0\in R$. This equals
\begin{align*}
t\left(x_{l+2}h_{l-1,l+1}+x_l^{a_l-1}h_{l,l-2}-x_{l+1}h_{l-1,l+2}\right).
\end{align*}
As $t$ is a non-zero-divisor in $R$, we have
\begin{align*}
x_{l+2}h_{l-1,l+1}+x_l^{a_l-1}h_{l,l-2}-x_{l+1}h_{l-1,l+2}=0\in R.
\end{align*}
Any parts of $h_{l-1,l+1}$ quasi-divisible by $x_k$ give terms in the above equation that are constant in $x_{l+1}$. However, all terms in $x_l^{a_l-1}h_{l,l-2}-x_{l+1}h_{l-1,l+2}$ are constant in all $x_j$ with $j\geq l+2$ or quasi-divisible by $x_{l+1}$. Therefore, the $x_k$-quasi-divisible terms of $h_{l-1,l+1}$ sum up to 0, a contradiction.

A symmetric argument shows that any $t$-constant term of $h_{i-1,i+1}$ is constant in $x_j$ for $j<i-1$, which finishes the proof of this lemma.
\end{proof}

The lemma above is the first step in an inductive argument for the following.
\begin{theorem}\label{thm:shifting}
After suitable isomorphism, the $g_{i-1,i+1}'$ are of the form
\begin{equation*}
g_{i-1,i+1}'=x_{i-1}(x_{i+1}+tc_{i+1})-x_i^{a_i}+tx_ih_i
\end{equation*}
for some $c_{i+1}\in S$ and $h_i\in S[x_i]$, $\deg_{x_i} h_i < a_i$. 
\end{theorem}
\begin{proof}
The previous lemma shows that it is possible to find an isomorphism such that this holds up to $t$-degree 1. Let us assume that we have found an isomorphism such that the equation holds up to $t$-degree $m$ for all $i$. We then fully shift both ways and by arguments analog to the $t$-degree 1 case we now have that we can write
\begin{equation*}
g_{i-1,i+1}'=x_{i-1}(x_{i+1}+tc_{i+1})-x_i^{a_i}+tx_ih_i+t^{m+1}H_i+t^{m+2}H_{i-1,i+1},
\end{equation*}
where $c_{i+1}$ and $h_i$ are as above and $H_i$ is constant in $x_{i-1}$, $x_{i+1}$ and $t$. We want to show $H_i=0$. As monomials in $H_i$ cannot be multiples of both, some $x_j$ with $j< i-1$ and some $x_l$ with $l>i+1$, we can split $H_i$ into $H_i^-$ and $H_i^+$, containing all monomials divisible by some $x_j$ with $j<i-1$ and $j>i+1$ respectively. Assume now that $H_l^+\neq 0$ for some $l$ and let $l$ be maximal in that regard. We again consider
\begin{align*}
x_{l+2}h_{l-1,l+1}+x_l^{a_l-1}h_{l,l+2}-x_{l+1}h_{l-1,l+2}=0\in R,
\end{align*}
where w.l.o.g. we assume that $h_{l-1,l+1}$, $h_{l,l+2}$ and $h_{l-1,l+2}$ are in compact form. We know that up to $t$-degree $m-1$ both $x_{l+2}h_{l-1,l+1}$ and $x_l^{a_l-1}h_{l,l+2}$ are constant in $x_j$ for $j\geq l+3$ and $j<l-1$, which implies that this also holds for $x_{l+1}h_{l-1,l+2}$.

We need to take a look at what happens when compacting the equation above up to $t$-degree $m-1$. First, $x_{l+2}h_{l-1,l+1}$ may need compacting by replacing $x_{l+2}x_{l-1}$ with $x_l^{a_l-1}x_{l+1}^{a_{l+1}-1}-th_{l-1,l+2}$. Next, $x_l^{a_l-1}h_{l,l+2}$ is already in compact form. Finally, when compacting $x_{l+1}h_{l-1,l+2}$ we might need to replace $x_{l-1}x_{l+1}$ with $x_l^{a_l}+th_{l-1,l+1}$. In $t$-degree $m$, this can only add terms to the equation that are present in $h_{l-1,l+1}$ and $h_{l-1,l+2}$ in $t$-degrees less than $m$. Any such terms in $h_{l-1,l+1}$ must be constant in all $x_j$ for $j\geq l+1$, and all such terms in $h_{l-1,l+2}$ are constant in all $x_j$ for $j\geq l+3$ and terms quasi-divisible by $x_{l+2}$ must also be quasi-divisible by $x_l$ (their compact form can't contain $x_{l+2}$ as the compact forms of $x_{l+2}h_{l-1,l+1}$ and $x_l^{a_l-1}h_{l,l+2}$ don't). This tells us that compacting the terms of $t$-degree less than $m$ of the equation doesn't add any term to the equation of $t$-degree $m$, that is constant in $x_{l+1}$ but non-constant in some $x_j$ with $j\geq l+2$.

Therefore, after completely compacting the equation, and then looking at the $t$-degree $m$ terms, we get that the only terms divisible by some $x_j$ with $j\geq l+2$ and constant in $x_{l+1}$ come from $x_{l+2}H_l^+$. This implies $x_{l+2}H_l^+=0$ and therefore
\begin{equation*}
H_l^+=0\in R,
\end{equation*}
in contradiction to $H_l^+\neq 0$.

Again, by symmetry we also have $H_i^-=0$ for all $i=2,\dots,e-1$.

We have therefore shown that after suitable isomorphism, the equation
\begin{equation*}
g_{i-1,i+1}'=x_{i-1}(x_{i+1}+tc_{i+1})-x_i^{a_i}+tx_ih_i
\end{equation*}
holds up to $t$-degree $m+1$. This finishes the induction step.

As $R$ is complete this process converges to give us the wanted isomorphism under which

\begin{equation*}
g_{i-1,i+1}'=x_{i-1}(x_{i+1}+tc_{i+1})-x_i^{a_i}+tx_ih_i.
\end{equation*}

\end{proof}

Now that we know very precisely how the $g_{i-1,i+1}'$ look like, we can also say something about the $g_{i,j}'$ with $j-i>2$.

\begin{corollary}
The ideal $I'$ contains elements of the form
\begin{align*}
g_{i,j}'=x_i(x_j+tc_j)-x_{i+1}^{a_{i+1}-1}\left(x_{i+2}^{a_{i+2}-2}\cdots x_{j-2}^{a_{j-2}-2}\right)x_{j-1}^{a_{j-1}-1} + th_{i,j},
\end{align*}
where $h_{i,j}\in S\llbracket  x_{i+1},\dots,x_{j-1}\rrbracket $.
\end{corollary}
\begin{proof}
We will use induction on $j-i$. For $j-i=2$, this is just the previous theorem.

Let it be known for all $g_{i,j}$ with $j-i<m$. First, it is clear that elements of this form exist in $I'$ for some $h_{i,j}\in S\llbracket  x_1,\dots,x_e\rrbracket $, as elements in $I$ lift to elements in $I'$. We can also assume that $h_{i,j}$ is in compact form. Let $j-i=m$ and consider
\begin{align*}
R\ni 0 &= x_{j-2}g_{i,j}'-x_ig_{j-2,j}'-x_{j-1}^{a_{j-1}-1}g_{i,j-1}'\\
&= t\left( x_{j-2}h_{i,j}-x_ih_{j-2,j}-x_{j-1}^{a_{j-1}-1}h_{i,j-1} \right)\\
&= x_{j-2}h_{i,j}-x_ih_{j-2,j}-x_{j-1}^{a_{j-1}-1}h_{i,j-1},
\end{align*}
i.e.
\begin{align*}
x_{j-2}h_{i,j}=x_i\underbrace{h_{j-2,j}}_{=x_{j-2}c_j+x_{j-1}h_{j-1}}+x_{j-1}^{a_{j-1}-1}\underbrace{h_{i,j-1}}_{\in S\llbracket  x_i,\dots,x_{j-2}\rrbracket }.
\end{align*}
Compacting the right side of this equation leaves us with monomials in $S\llbracket  x_i,\dots,x_{j-1}\rrbracket $, where any monomial containing $x_{j-1}$ is either $x_{j-1}^{a_{j-1}-1}$ or also divisible by $x_{j-2}$. On the other hand, if $h_{i,j}$ would contain any term divisible by some $x_l$ with $l\geq j+1$, then after compacting $x_{j-2}h_{i,j}$ we would still have a term divisible by some $x_{l'}$, with $j\leq l'\leq l$. If $h_{i,j}$ contains a monomial constant in $x_l$ for all $l\geq j+1$ but divisible by $x_j$, then after compacting $x_{j-2}h_{i,j}$ would contain a term that is a multiple of $x_{j-1}$ but constant in $x_{j-2}$. This shows that $h_{i,j}$ must be constant in all $x_{l}$ for $l\geq j$.

After fully shifting to the right, we can make a symmetric argument to show that $h_{i,j}$ must be constant in all $x_{l}$ for $l\leq i$. Fully shifting back to the left adds the term $x_ic_j$ to $g_{i,j}'$, which we write separately to get the form claimed by the corollary. 
\end{proof}


\section{Singularities of the Deformation}\label{sec:sing_of_def}
We now come to the main result of this article.

In the previous section we established that every deformation of a toric surface singularity is isomorphic to a deformation of the form
\begin{align*}
S\llbracket  x_1,\dots,x_e\rrbracket /I',
\end{align*}
where $I'$ contains elements $g_{i-1,i+1}'$, $i=2,\dots,e-1$ of the form
\begin{equation*}
g_{i-1,i+1}' = x_{i-1}(x_{i+1}+tc_{i+1})-x_i^{a_i}+tx_ih_i
\end{equation*}
for some $c_{i+1}\in S$ and $h_i\in S[x_i]$, $\deg_{x_i} h_i < a_i$.

We will now use this standardized form to calculate the singularities of the generic fiber of the deformation.\\

\begin{theorem}\label{thm:toric_to_toric}
Let $X$ be surface with at worst toric singularities and let $\mathcal{X}$ be a proper deformation of $X$ over a complete DVR $S$. Then the geometric generic fiber of $\mathcal{X}$ over $S$ contains at worst toric surface singularities.
\end{theorem}
\begin{proof}
By \ref{thm:proper_to_sing_def}, it is enough to show that singularity-deformations of toric surface singularities contain at worst toric surface singularities. We thereby assume $X$ to be a toric surface singularity and $\mathcal{X}$ a singularity deformation of $X$ over $S$.

As all schemes are affine, this can just as well be read as a statement in commutative algebra. Let $R:=S\llbracket x_1,\dots,x_e\rrbracket/I\otimes_S K$ be the ring corresponding to the geometric generic fiber and let $\p=(x_1-\lambda_1,\dots,x_e-\lambda_e)\subseteq R$ be a maximal ideal.

We are interested in the singularity
\begin{align*}
\widehat{R_\p}\cong K\llbracket x_1-\lambda_1,\dots,x_e-\lambda_e\rrbracket/I'.
\end{align*}

Let $i_e$ be the minimal index such that $\lambda_{i_e}\neq 0$, meaning that $x_{i_e}$ becomes invertible in $\widehat{R_\p}$. Then by the previously established form of singularity-deformations of toric surface singularities (see \ref{thm:shifting}), for all $j\geq i_e+2$ we have that
\begin{align*}
x_j &= x_j-x_{i_e}^{-1}g_{i_e,j}'\\
&= x_j-x_{i_e}^{-1}\left(x_{i_e}(x_j+tc_j)-x_{{i_e}+1}^{a_{i+1}-1}\left(x_{{i_e}+2}^{a_{{i_e}+2}-2}\cdots x_{j-2}^{a_{j-2}-2}\right)x_{j-1}^{a_{j-1}-1} + th_{{i_e},j}\right)\\
&= tc_j-x_{i_e}^{-1}x_{{i_e}+1}^{a_{{i_e}+1}-1}\left(x_{{i_e}+2}^{a_{{i_e}+2}-2}\cdots x_{j-2}^{a_{j-2}-2}\right)x_{j-1}^{a_{j-1}-1} + x_{i_e}^{-1}th_{{i_e},j}\in \widehat{R_\p}.
\end{align*}
As we can write $x_j$ in terms of $x_{i_e},\dots,x_{j-1}$, for all $j\geq i_e+2$, projecting onto the variables $x_1,\dots,x_{i_e+1}$ is an isomorphism. This means that we have
\begin{align*}
\widehat{R_\p}\cong K\llbracket x_1,\dots,x_{i_e-1},x_{i_e}-\lambda_{i_e},x_{i_e+1}-\lambda_{i_e+1}\rrbracket /I'.
\end{align*}

Similarly, let $i_1$ be the maximal index such that $\lambda_{i_1}\neq -tc_{i_1}$ or $x_{i_1}+tc_{i_1}\notin\p$. Then we can write every $x_j$ with $j\leq i_1-2$ in terms of $x_{j+1},\dots,x_{i_1}$.

If $i_1>i_e$, then we can write every $x_j$ in terms of $x_{i_1-1},x_{i_1}$ (or alternatively $x_{i_e},x_{i_e+1}$) and as there is no relation between them we are in the smooth case. Otherwise, we have $0=\lambda_i=-tc_i$ for all $i_1<i<i_e$, and
\begin{align*}
\widehat{R_\p}\cong K\llbracket x_{i_1-1},\dots,x_{i_e-1},x_{i_e}-\lambda_{i_e},x_{i_e+1}-\lambda_{i_e+1}\rrbracket /I',
\end{align*}
where $I'$ contains elements
\begin{align*}
x_{i-1}x_{i+1}-x_i^{a_i}+tx_ih_i
\end{align*}
for $i=i_1,\dots,i_e-2$ and
\begin{align*}
x_{i-1}(x_{i+1}-\lambda_{i+1})-x_i^{a_i}+tx_ih_i
\end{align*}
for $i=i_e-1,i_e$.

After a change of variable (or shifting to the right at $i_e-1$ and $i_e$), we can additionally assume $\lambda_{i_e}=\lambda_{i_e+1}=0$.\footnote{This will also change $h_{i_e}$ slightly and we might now have $h_{i_e}\in K[x_{i_e}]$ instead of $h_{i_e}\in S[x_{i_e}]$, but we need not take note of that.}

We have reduced to the case
\begin{align*}
\widehat{R_\p}\cong K\llbracket x_{i_1-1},\dots,x_{i_e+1}\rrbracket/I'.
\end{align*}
Furthermore, $I'$ contains
\begin{align*}
x_{i-1}x_{i+1}-x_i^{a_i}+tx_ih_i=x_{i-1}x_{i+1}-x_i^{b_i}(x_i^{a_i-b_i}+t\frac{h_i}{x_i^{b_i-1}}),&&\forall i=i_1,\dots,i_e
\end{align*}
where $b_i-1$ is the lowest $x_i$-degree among monomials in $h_i$. The factor $x_i^{a_i-b_i}+t\frac{h_i}{x_i^{b_i-1}}$ becomes invertible when localizing away from $\p$. Therefore, sending $x_{i+1}$ to $x_{i+1}\left(x_i^{a_i-b_i}+t\frac{h_i}{x_i^{b_i-1}}\right)$ is an isomorphism in $\widehat{R_\p}$. This might lead to some unwanted factors in neighboring equation, these can be dealt with by multiplying $x_{i+1}$ and $x_{i+2}$ with suitable powers of the same factor. Starting at $i=i_1$ we can use this to remove all invertible factors from the equations and are left with
\begin{align*}
0=x_{i-1}x_{i+1}-x_i^{b_i}\in \widehat{R_\p}
\end{align*}
for $i=i_1,\dots,i_e$. Applying theorem \ref{thm:toric_eq}, we see that $\widehat{R_\p}$ is either smooth or a toric surface singularity.
\end{proof}

Close inspection of the proof of Theorem \ref{thm:toric_to_toric} reveals the following semicontinuity.

\begin{proposition}\label{prop:semicontinuity}
Let $y$ be a singularity of the generic fiber and \linebreak $x=X(a_1,\dots,a_e)$ an isolated toric surface singularity of the special fiber, such that $x$ is a specialization of $y$. Then $y$ is an isolated toric surface singularity, so that there exist $e'\in \NN$ and $b_1,\dots,b_{e'}\in\NN$ with $y\cong X(b_1,\dots,b_{e'})$. The proof above shows that we may choose $e'$ and the $b_i$ such that
\begin{itemize}
\item $e'\leq e$ \qquad\qquad and
\item $b_i\leq a_{i+l}$
\end{itemize}
for some $0\leq l\leq e-e'$ and all $i=1,\dots, e'$.

When seeing $x$ and $y$ as quotient singularities by $\mu_n$ and $\mu_{n'}$ respectively, this implies $n'\leq n$. This means if $x$ deforms to $y$, then the length of the associated group scheme to $y$ is less than or equal to that of $x$. 
\end{proposition}

As an application of the theorem, we obtain Riemenschneider's conjecture on isolated cyclic quotient singularities in full generality.

\begin{theorem}\label{thm:cyc_to_cyc}
Let $X$ be an isolated cyclic lrq singularity and let $\mathcal{X}$ be a proper deformation of $X$ over a complete DVR $S$. Then the geometric generic fiber $\mathcal{X}_{\overline{\eta}}$ of $\mathcal{X}\to \Spec S$ contains at worst isolated cyclic lrq singularities.
\end{theorem}
\begin{proof}
In dimension $\geq 3$, this is true by rigidity of quotient singularities. The characteristic zero case is due to \cite{Sch71}. The case of positive and mixed characteristic was proven in \cite{LMM25}.

In characteristic zero and dimension 2, this is a result in \cite{KSB88}.

Finally, theorem \ref{thm:toric_to_toric} together with proposition \ref{prop:toric_eq_cyclic} proves the case of dimension 2 in full generality, establishing the conjecture in positive and mixed characteristic, and providing a new proof for characteristic zero. 
\end{proof}

\bibliographystyle{alpha}
\bibliography{literature}

\begin{thebibliography}{LMM25}

\bibitem[EV85]{EV85}
Hélène Esnault and Eckart Viehweg.
\newblock Two dimensional quotient singularities deform to quotient
  singularities.
\newblock {\em Mathematische Annalen}, 271:439--450, 1985.

\bibitem[KSB88]{KSB88}
J{\'a}nos Koll{\'a}r and Nicholas~Ian Shepherd-Barron.
\newblock Threefolds and deformations of surface singularities.
\newblock {\em Inventiones mathematicae}, 91(2):299--338, Jun 1988.

\bibitem[LMM25]{LMM25}
Christian Liedtke, Gebhard Martin, and Yuya Matsumoto.
\newblock Isolated quotient singularities in positive characteristic.
\newblock {\em Astérisque}, 461, 2025.

\bibitem[Rie74]{Rie74}
Oswald Riemenschneider.
\newblock Deformationen von quotientensingularitäten (nach zyklischen
  gruppen).
\newblock {\em Mathematische Annalen}, 209:211--248, 09 1974.

\bibitem[Sch71]{Sch71}
Michael Schlessinger.
\newblock Rigidity of quotient singularities.
\newblock {\em Inventiones mathematicae}, 14(1):17--26, Mar 1971.

\bibitem[ST25]{ST21}
Kenta Sato and Shunsuke Takagi.
\newblock Arithmetic and geometric deformations of {$F$}-pure and {$F$}-regular
  singularities.
\newblock {\em Amer. J. Math.}, 147(2):561--596, 2025.

\end{thebibliography}

\end{document}